\documentclass[12pt,reqno]{amsart}

\usepackage{amscd}

\usepackage{amsthm}
\usepackage{amstext}
\usepackage{amssymb}
\usepackage{amsmath,amscd}
\usepackage{enumerate}
\usepackage{mathrsfs}
\usepackage{amsrefs}
\usepackage[utf8]{inputenc}
\usepackage{geometry}
\usepackage{xcolor}
\usepackage{adjustbox}
\usepackage[all]{xy}
\usepackage{stmaryrd}
\usepackage{booktabs}

\numberwithin{equation}{section}
\newcommand{\Fr}{\textnormal{Fr}}

\newtheorem{theorem}{Theorem}[section]
\newtheorem{proposition}[theorem]{Proposition}
\newtheorem{lemma}[theorem]{Lemma}
\newtheorem{corollary}[theorem]{Corollary}
\newtheorem{conjecture}[theorem]{Conjecture}

\theoremstyle{definition}
\newtheorem{definition}[theorem]{Definition}
\newtheorem{remark}[theorem]{Remark}
\newtheorem{example}[theorem]{Example}

\newcommand{\cB}{\mathcal B}

\newcommand{\C}{\mathbb C}
\newcommand{\F}{\mathbb F}

\newcommand{\N}{\mathbb N}

\newcommand{\Z}{\mathbb Z}

\title{Invertible Calabi-Yau Orbifolds over Finite Fields}

\author{Marco Aldi}
\address{Department of Mathematics and Applied Mathematics\\
Virginia Commonwealth University\\
Richmond, VA 23284, USA}

\author{Andrija Peruni\v{c}i\'{c}}
\address{}

\begin{document}

\begin{abstract}
In the context of Berglund-H\"ubsch mirror symmetry, we compute the eigenvalues of the Frobenius endomorphism acting on a $p$-adic version of Borisov's complex. As a result, we conjecture an explicit formula for the number of points of crepant resolutions of invertible Calabi-Yau orbifolds defined over a finite field. 
\end{abstract}

\maketitle

\section{Introduction}

A generalization of the Greene-Plesser example \cite{GP}, the Berglund-H\"ubsch construction \cite{BH} produces pairs of mirror Calabi-Yau orbifolds of hypersurfaces in weighted projective spaces. The explicit case-by-case isomorphism of these mirror models was first proved by Kreuzer \cite{Kre} and, later and in greater generality, by Krawitz  \cite{Kra}. A conceptual proof of mirror symmetry for Berglund-H\"ubsch pairs was given by Borisov \cite{B1} using vertex algebras. Even though the physically natural ground field for Berglund-H\"ubsch mirror symmetry is that of the complex numbers, the construction works over arbitrary fields. Over the field of $p$-adic numbers an arithmetically meaningful Frobenius operator can be constructed \cite{P} at the level of cohomology, but unfortunately it admits no obvious lift to the level of cochains. Motivated by this problem, a deformation of Borisov's complex was constructed in \cite{AP} (using a D-module theoretic approach inspired by \cite{SchwarzShapiro09}) with the property that when specialized to the $p$-adic numbers it would indeed carry a natural cochain-level Frobenius operator. Furthermore, it was shown in \cite{AP} that the Frobenius operators acting on deformed Borisov complexes corresponding to Berglund-H\"ubsch mirror pairs are intertwined by mirror symmetry.

In this paper we continue the line of inquiry started in \cite{AP} and further investigate the arithmetic information stored in the $p$-adic deformed Borisov complex. Both the Borisov complex and its deformation introduced in \cite{AP} are constructed out of polynomials in two sets of variables, corresponding to the variables of the superpotentials on the two mirror models. While the Frobenius operator used in \cite{AP} acts on all the variables, in this paper we restrict it to only one set of variables. Additionally, we replace the standard Frobenius operator with its more arithmetically natural inverse (see e.g.\ \cite{M}). Finally, following earlier work on the (\'etale) cohomology of orbifolds \cite{Behrend93,Rose07}, we twist the Frobenius operator by a suitable power of $p$ depending on the notion of age. Our first observation is that if the determinant of the matrix defining the superpotential divides $p-1$, then our modified Frobenius operator is diagonalized by the standard basis of the total cohomology of the deformed Borisov bicomplex. Furthermore, we explicitly express the eigenvalues of the cohomological Frobenius in terms of $p$-adic gamma functions. 

Drawing motivation from the LG/CY correspondence \cite{ChiodoRuan11} and from methods for calculating hypersurface point counts over finite fields via traces of the Frobenius endomorphism acting on the Monsky-Washnitzer cohomology, we form a conjecture: the number of points of any crepant resolution of a Berglund-H\"ubsch orbifold can be computed by the supertrace of our modified Frobenius operator acting on the cohomology of the deformed $p$-adic Borisov complex associated to the orbifold. As evidence for our conjecture, we first point out that it correctly counts points modulo $p$. We then prove the conjecture for all Berglund-H\"ubsch elliptic curves (including those expressed in terms of chains, loops, and their combinations), and for all weighted diagonal K3 surfaces. In the latter case, our orbifold calculations recover the results obtained by Goto \cite{Goto} through a careful analysis of the corresponding minimal resolutions of singularities.

\section{Preliminaries}
In this section we recall the basics of the $p$-adic Berglund-H\"ubsch construction. We refer to \cite{AP} for additional detail.

\begin{definition}\label{def:2.1}
Let $A$ be an invertible $n\times n$ matrix $A$ with non-negative entries and consider the polynomial
\[
W_A(x)=\sum_{i=1}^n x^{e_i A}
\]
where $e_1,e_2,\ldots,e_n$ are the elements of the standard basis of $\mathbb Z^n$ (throughout the paper, we assume our vectors in $\mathbb Z^n$ to be row vectors). We say that $A$ is {\it Berglund-H\"ubsch} over a prime field $\mathbb F_p$ if the following conditions are met:
\begin{enumerate}[1)]
\item $W_A$ is quasi-homogeneous;
\item $(\partial_{x_1}W_A,\ldots, \partial_{x_n}W_A)$ is a regular sequence in $\mathbb F_p[x]=\mathbb F_p[x_1,\ldots,x_n]$;
\item $\det(A)$ divides $p-1$.
\end{enumerate}
\end{definition}

\begin{remark}
Conditions 1) and 2) in Definition \ref{def:2.1} can be summarized as saying that the polynomial $W_A$ is {\it invertible}. While condition 3) is admittedly strong, it does hold for infinitely many primes once the integer matrix $A$ is fixed. While for simplicity we work with fields of prime cardinality $p$, most of our arguments can be adapted to the case of cardinality $p^r$.
\end{remark}

\begin{example}
The matrix $A=\left(\begin{smallmatrix} 2 & 1 \\ 0 &3\end{smallmatrix}\right)$ is Berglund-H\"ubsch over $\mathbb F_7$. The corresponding polynomial is $W_A(x)=x_1^2x_2+x_2^3$.
\end{example}

\begin{remark}
$A$ is Berglund-H\"ubsch over $\mathbb F_p$ if and only if $A^T$ is Berglund-H\"ubsch over $\mathbb F_p$.
\end{remark}

\begin{definition}
Let $A$ be an $n\times n$ Berglund-H\"ubsch matrix over $\mathbb F_p$. Fix a $\det(A)$-th root of unity $\omega$ and consider the scaling action of $\mathbb Z^n$ on $\mathbb F_p[x]$ extended by linearity from $\mu\cdot x^\gamma = \omega^{\gamma\mu^T} x^\gamma$. Under this action, the stabilizer of $W_A(x)$ is canonically identified with $G_A=\mathbb Z^n/\mathbb Z^nA^T$.
\end{definition}

\begin{example}
The matrix $A=\left(\begin{smallmatrix} 2 & 1 & 0\\ 0 & 2 & 1\\ 1& 0 & 3\end{smallmatrix}\right)$ is Berglund-H\"ubsch over $\mathbb F_{53}$ and $G_A$ is isomorphic to the cyclic group of order $13$ generated by $e_1$.
\end{example}

\begin{proposition}[\cite{Kre}]
Let $A$ be a Berglund-H\"ubsch matrix over a field $\mathbb F_p$. If $A$ is irreducible (i.e.\ cannot be written as a direct sum of Berglund-H\"ubsch matrices of smaller size) then, up to permutation of the variables, $W_A(x)$ is either a {\it chain} (i.e.\ of the form $x_1^{a_1}x_2+x_2^{a_2}x_3+\cdots+x_{n-1}^{a_{n-1}}x_n+x_n^{a_n}$) or a {\it loop} (i.e.\ of the form $x_1^{a_1}x_2+x_2^{a_2}x_3+\cdots+x_{n-1}^{a_{n-1}}x_n+x_n^{a_n}x_1$).
\end{proposition}

\begin{example}
A {\it diagonal} polynomial of the form $W_A(x)=\sum_{i=1}^n x_i^{a_i}$ with $A={\rm diag}(a_1,\ldots,a_n)$ can be thought of as a direct sum of one-dimensional chains.
\end{example}

\begin{definition}
Let $A$ be Berglund-H\"ubsch matrix over $\mathbb F_p$ and let $A^{-T}$ be its inverse transpose. Let $\widetilde \cB_A$ be the subspace of overconvergent series in the even variables $\{x_i,y_i\}_{i=1}^n$ and in the odd variables $\{\theta_i\}_{i=1}^n$ spanned by all monomials $x^\gamma y^\lambda \theta^I$ with the property that $\lambda A^{-T}\ge 0$. Let $\cB_A$ be the quotient of $\widetilde \cB_A$ by the ideal generated by the monomials $x^\gamma y^\lambda \theta^I$ such that $\gamma A^{-1}\lambda^T>0$. Let $Q_A^\vee$ be the {\it grading} operator acting on $\cB_A$ with respect to which each monomial $x^\gamma y^\lambda \theta^I$ is an eigenvector with eigenvalue $q^\vee(\lambda,I)$ equal to the number of indices $i\in \{1,\ldots,n\}$ such that $(\lambda A^{-T})_i>(\lambda A^{-T})_iI_i$. Similarly, the {\it dual grading operator} $Q_A$ has the property that each monomial $x^\gamma y^\lambda \theta^I$ is an eigenvector with eigenvalue
\[
q(\lambda,I)=I_1+\cdots+I_n+q^\vee(\lambda,I).
\]
For a fixed $\pi\in \C_p$ such that $\pi^{p-1}=-p$, we define the operators
\[
d_A(x^\gamma y^\lambda \theta^I)=\sum_{i=1}^n\gamma_i x^\gamma y^\lambda \theta^{I+e_i} + \pi\sum_{i,j=1}^n A_{ji}\, x^{\gamma + e_j A}y^\lambda \theta^{I+e_i}
\]
and
\[
d_{A}^\vee(x^\gamma y^\lambda \theta^I)=\sum_{i=1}^n\pi^{-1}(\lambda A^{-T})_i\, x^\gamma y^\lambda \partial_{\theta_i} \theta^I + \sum_{i=1}^nx^\gamma y^{\lambda + e_iA^T} \partial_{\theta_i} \theta^I
\]
acting on $\mathcal B_A$.
\end{definition}

\begin{proposition}[\cite{AP}]\label{prop:2.8}
Let $A$ be a Berglund-H\"ubsch matrix over $\mathbb F_p$. Then:
\begin{enumerate}[1)]
    \item $(\mathcal B_A,d_A,d^\vee_A)$ is a bicomplex with respect to the bigrading induced by $Q_A$ and $Q_A^\vee$.

    \item $(\cB_A,d_A,d_A^\vee)$ decomposes into the direct sum of sub-bicomplexes $(\cB_A(\gamma,\lambda),d_A,d_A^\vee)$ labeled by classes $[(\gamma,\lambda)]\in \Z^{2n}/(\Z^{2n}(A\oplus A^T))\cong G_{A^T}\times G_A$.
\item For each $(\gamma,\lambda)\in G_{A^T}\times G_A$, the total cohomology of $(\cB_A(\gamma,\lambda),d_A,d_A^\vee)$ is at most one dimensional.
\item (Unprojected Duality) For every $(\gamma,\lambda)\in G_{A^T}\times G_A$, the cohomology of the cochain complex $(\cB_A(\lambda,\gamma),d_A+d^\vee_A)$ is isomorphic to the cohomology of the cochain complex $(\cB_{A^T},d_{A^T}+d_{A^T}^\vee)$.
\end{enumerate}
\end{proposition}

\begin{definition}\label{def:2.9}
Let $A$ be a Berglund-H\"ubsch matrix over $\mathbb F_p$ and let $G$ be a subgroup of $G_A$. The {\it transpose} of $G$ is the subgroup $G^T$ of $G_{A^T}$ defined by requiring that an element $\gamma\in G_{A^T}$ belongs to $G^T$ if and only if $\gamma A^{-1} \lambda^T \in \mathbb Z$ for all $\lambda\in G$.
\end{definition}

\begin{example}
The matrix $A=\left(\begin{smallmatrix} 2 & 1 & 0\\ 0 & 3 & 0\\ 0& 0 & 3\end{smallmatrix}\right)$ is Berglund-H\"ubsch over $\mathbb F_{19}$. $G_A$ is of order $18$ generated by $e_2$ and $e_3$. $G_{A^T}$ is also of order $18$ generated by $e_1$ and $e_3$. If $G$ is the cyclic subgroup of $G_A$ order $3$ generated by $e_3$, then $G^T$ is the cyclic subgroup of $G_{A^T}$ order $6$ generated by $e_1$.
\end{example}

\begin{definition}
Let $A$ be a Berglund-H\"ubsch matrix over $\mathbb F_p$. Given $\lambda\in G_A$, let $\mathcal B_A^\lambda$ be the direct sum of all $\mathcal B_A(\gamma,\lambda)$ over all $\gamma\in G_{A^T}$.
\end{definition}

\begin{remark}\label{rem:2.12}
Let $A$ be a Berglund-H\"ubsch matrix over $\mathbb F_p$. The dimension of the total cohomology group $H(\mathcal B_A^0,d_A+d_A^\vee)$ is equal to the dimension of the Milnor ring $\mathbb F_p[x]/dW_A$. For instance, if $A$ is diagonal, then the total cohomology group has dimension $(A_{11}-1)\cdots(A_{nn}-1)$.
\end{remark}

\begin{remark}\cite{AP}
Let $A$ be an $n\times n$ Berglund-H\"ubsch matrix over $\mathbb F_p$. Given $\lambda\in G_A$, let $A^\lambda$ be the matrix such that $W_{A^\lambda}(x)$ is obtained from $W_A(x)$ by setting $x_i=0$ whenever $(\lambda A^{-T})_i\in \mathbb Q\setminus \mathbb Z$. Then $A^\lambda$ is Berglund-H\"ubsch over $\mathbb F_p$ and $H(\mathcal B_{A^\lambda}^0,d_{A^\lambda}+d_{A^\lambda}^\vee)\cong H(\mathcal B_A^\lambda,d_A+d_A^\vee)$.
\end{remark}

\begin{definition}
Let $A$ be Berglund-H\"ubsch matrix over $\mathbb F_p$ and let $G$ be a subgroup of $G_A$. Let $\mathcal B_A^G$ be the direct sum of all $\mathcal B_A(\gamma,\lambda)$ over all $(\gamma,\lambda)\in G^T\times G$. The {\it projected double complex of $A$ with respect to $G$} is $(\mathcal B_A^G,d_A,d_A^\vee)$.
\end{definition}

\begin{remark}
It follows from Definition \ref{def:2.9} that $\mathcal B_A^G$ can be identified with the $G$-invariant part of  $\bigoplus_{\lambda\in G} \mathcal B_A^\lambda$.
\end{remark}

\begin{remark}
It follows from Proposition \ref{prop:2.8} that the cohomology of $(\mathcal B_A^G,d_A+d_A^\vee)$ is isomorphic to that of $(\mathcal B_{A^T}^{G^T},d_{A^T}+d_{A^T}^\vee)$.
\end{remark}

\begin{definition}\label{def:2.17}
Let $A$ be an $n\times n$ Berglund-H\"ubsch matrix over $\mathbb F_p$. Assume the representative in the equivalence class of $\lambda\in G_A$ modulo $\mathbb Z^nA^T$ is chosen so that $0\le (\lambda A^{-T})_i<1$ for every $i\in \{1,\ldots,n\}$. The {\it age} of $\lambda$ is ${\rm age}(\lambda)=\sum_{i=1}^n(\lambda A^{-T})_i$. Similarly, the {\it dual age} of $\gamma\in G_{A^T}$ is defined to be ${\rm age}^\vee(\gamma)=\sum_{i=1}^n(\gamma A^{-1})_i$. We also define $\dim(\lambda)$ to be the number of indices $i$ for which $(\lambda A^{-T})_i\in \mathbb Z$.
\end{definition}

\begin{remark}\label{rem:2.18}
The bicomplex $(\mathcal B_A,d_A,d_A^\vee)$ was introduced in \cite{AP} as a $p$-adic D-module counterpart of the $\mathbb C$-linear bicomplex introduced by Borisov in \cite{B1} (which can be formally recovered by setting $\pi=0$). As pointed out in Remark \ref{rem:2.12} the two bicomplexes have the same total cohomology. Since Borisov's bicomplex computes the same FJRW groups as in \cite{Kra}, we see that (up to a shift in bi-degree) the same holds for $(\mathcal B_A^G,d_A,d_A^\vee)$. More precisely, let $A$ be Berglund-H\"ubsch over $\mathbb F_p$. Then
\begin{equation}
\dim H^{r,s}_{{\rm FJRW}}(W_A,G)=\sum_{\gamma,\lambda} \dim H(\mathcal B_A(\gamma,\lambda),d_A+d_A^\vee)\,,
\end{equation}
where the sum is over all $\lambda\in G$ and $\gamma\in G^T$ such that $s={\rm age}(\lambda)+{\rm age}^\vee(\gamma)-1$ and $r=\dim(\lambda)+{\rm age}(\lambda)-{\rm age}^\vee(\gamma)-1$.
\end{remark}

\begin{remark}
In general, the age (and thus the FJRW bigrading) is not necessarily an integer valued. However, if both $G$ and $G^T$ contain $J=e_1+\ldots+e_n$ (which, in particular, implies the Calabi-Yau condition $JA^{-T}J^T\in \mathbb Z$), then the age of any $\lambda\in G$ and $\gamma\in G^T$ is an integer.
\end{remark}

\begin{remark}\label{rem:2.20}
Let $A$ be an $n\times n$ Berglund-H\"ubsch matrix over $\mathbb F_p$ and let $G$ be a subgroup of $G_A$ such that $J\in G\cap G^T$. Let $X_A(\mathbb C)$ be the hypersurface in complex weighted projective space with weights $w_i=m(JA^{-T})_i$, where $m$ is the smallest integer such that each $w_i$ is an integer, cut out by $W_A(x)=0$. We observe that the subgroup of $G$ generated by $J$ acts trivially on $X_A(\mathbb C)$. Then it follows from \cite{ChiodoRuan11} that the Chen-Ruan Hodge numbers of the orbifold $Y_{A,G}(\mathbb C)=[X_A/(G/\langle J\rangle)](\mathbb C)$ are given by
\begin{equation}\label{eq:2.2}
H^{r,s}_{{\rm CR}}(Y_{A,G}(\mathbb C))= \sum_{\gamma,\lambda} \dim H(\mathcal B_A(\gamma,\lambda),d_A+d_A^\vee)\,,
\end{equation}
where the sum is over all $\lambda\in G$ and $\gamma\in G^T$ such that $s={\rm age}(\lambda)+{\rm age}^\vee(\gamma)-1$ and $r=\dim(\lambda)+{\rm age}(\lambda)-{\rm age}^\vee(\gamma)-1$. Assume furthermore that $JA^{-T}J^T=1$ (strict Calabi-Yau condition) and that $Y_{A,G}(\mathbb C)$ admits a crepant resolution $Z_{A,G}(\mathbb C)$. Then the RHS of \eqref{eq:2.2} calculates the dimension of the (ordinary) Hodge group of $Z_{A,G}(\mathbb C)$ in bidegree $(r,s)$.
\end{remark}

\section{Frobenius Maps}

\begin{definition}
Let $A$ be an $n\times n$ Berglund-H\"ubsch matrix over $\mathbb F_p$. The {\it Frobenius map} $\Fr_A$ is the linear endomorphism of $\mathcal B_A$ that to each monomial $x^\gamma y^\lambda \theta^I \in \cB_A$ associates the overconvergent power series
\begin{equation}\label{eq:3.1}
\Fr_A(x^\gamma y^\lambda \theta^I)=p^{-q(\lambda,I)} \sum c_k \pi^{|k|} x^{(\gamma+kA)/p}y^\lambda \theta^I\,,
\end{equation}
where the sum is over all $k=(k_1,\ldots,k_n)\in \N^n$ such that $\gamma+kA\in p\N^n$ and the coefficients $c_k$ are defined recursively by setting $c_0=1$, $c_k=0$ if $k_j<0$ for some $j$, and
\[
k_jc_k=\pi c_{k-e_j}+\pi^p c_{k-pe_j}
\]
otherwise. Similarly, the {\it dual Frobenius map} $\Fr^\vee_A$ is defined by the formula
\[
\Fr^\vee_A (x^\gamma y^\lambda \theta^I)= p^{-q^\vee(\lambda,I)}\sum c_k \pi^{|k|} x^\gamma y^{(\lambda+kA^T)/p} \theta^I\,,
\]
where now $k$ ranges over all elements of $\N^n$ such that $\lambda+kA^T\in p\N^n$.
\end{definition}

\begin{remark}
${\rm Fr}_A$ and ${\rm Fr}_A^\vee$ are, respectively, the inverses of the operators denoted by ${\rm Fr}_A'$ and ${\rm Fr}_A''$ in \cite{AP}. In \cite{AP}, the composition ${\rm Fr}_A''{\rm Fr_A}'$ was shown to commute (up to an overall power) with Berglund-H\"ubsch duality. Here we use inverse operators for the purpose of calculating the number of points.
\end{remark}

\begin{lemma}
Let $A$ be an $n\times n$ Berglund-H\"ubsch matrix over $\mathbb F_p$. Then $\Fr_A$ and $\Fr^\vee_A$ are endomorphisms of the bicomplex $(\cB_A, d_A,d_A^\vee)$.
\end{lemma}

\begin{proof} $d_A$ commutes with ${\rm Fr}^\vee_A$ because they act on different sets of even variables and $d_A$ preserves the $Q_A^\vee$ grading. Similarly, $d_A^\vee$ commutes with ${\rm Fr}_A$. Consider the factorization ${\rm Fr}_A={\rm Fr_1}\circ{\rm Fr_2}$ where
\[
{\rm Fr_2}(x^\gamma y^\lambda \theta^I)=\sum_k c_k \pi^{|k|}x^{\gamma+kA}y^\lambda \theta^I
\]
and
\[
{\rm Fr_1}(x^\gamma y^\lambda \theta^I)=
\begin{cases}
p^{-q(\lambda,I)} x^{\gamma/p}y^\lambda \theta^I &\textrm{ if } \gamma\in p\mathbb N^n\\
0 & \textrm{otherwise}\,.
\end{cases}
\]
To show that ${\rm Fr}_A$ commutes with $d_A$ it suffices to show that $d_{pA}\circ {\rm Fr}_2={\rm Fr}_2\circ d_A$ and $d_A\circ {\rm Fr}_1={\rm Fr}_1 \circ d_{pA}$. $d_A$ decomposes as the sum of an operator $T$ (which does not depend on $A$) that acts by derivation on the even variables and an operator $\delta_A$ which multiplies the argument by a polynomial. Since ${\rm Fr}_2$ multiplies the argument by an overconvergent power series we have
\begin{align}
[T,{\rm Fr}_2]&=\sum_{k}\sum_i c_k\pi^{|k|}(kA)_ix^{kA}\label{eq:3.2}\\
&=\sum_k\sum_{i,j} c_k \pi^{|k|}k_jA_{ji}x^{kA}\label{eq:3.3}\\
&=\pi\sum_k\sum_{i,j}c_k\pi^{|k|}A_{ji}x^{kA+e_jA}-\pi\sum_k\sum_{i,j}c_k\pi^{|k|}(pA_{ji})x^{kA+e_j(pA)}\label{eq:3.4}\\
&={\rm Fr}_2\circ\delta_A-\delta_{pA}\circ{\rm Fr}_2\,,
\end{align}
where the RHS of \eqref{eq:3.2}-\eqref{eq:3.4} denote multiplication by the corresponding overconvergent power series in $x$. Here we use the notation $|k|=k_1+\cdots+k_n$ for $k=(k_1,\ldots,k_n)$. This implies $d_{pA}\circ {\rm Fr}_2={\rm Fr}_2\circ d_A$. Suppose that $\gamma\notin p\mathbb N^n$. Then $\gamma+e_jpA\notin p\mathbb N^n$ and thus
\[
{\rm Fr}_1\circ d_{pA}(x^\gamma y^\lambda \theta^I)=0=d_A\circ {\rm Fr}_1(x^\gamma y^\lambda \theta^I)\,.
\]
On the other hand, if $\gamma\in p\mathbb N^n$, then
\begin{align}
d_A\circ {\rm Fr}_1(x^\gamma y^\lambda \theta^I)&= p^{-q(\lambda,I)-1}\left(\sum_{i=1}^n\gamma_i x^{\gamma/p}y^\lambda \theta^{I+e_i}+\pi \sum_{i,j}pA_{j,i} x^{\gamma/p+e_jA}y^\lambda \theta^{I+e_i}\right)\\
&=\sum_{i=1}^n p^{-q(\lambda,I+e_i)}\gamma_i x^{\gamma/p}y^\lambda \theta^{I+e_i}+\pi \sum_{i,j} p^{-q(\lambda,I+e_i)}pA_{j,i}y^\lambda \theta^{I+e_i}\\
&={\rm Fr}_1\circ d_{pA}(x^\gamma y^\lambda \theta^I)\,,
\end{align}
where we used the fact that $d_{pA}$ is of degree $1$ with respect to the $Q_A$ grading. This proves that ${\rm Fr}_A$ is an endomorphism of the bicomplex $(\mathcal B,d_A,d^\vee_A)$. The remaining statement, namely that $d_A^\vee$ commutes with ${\rm Fr}^\vee_A$\,, is proved in a similar way. First we introduce the factorization ${\rm Fr}^\vee_A={\rm Fr}_1^\vee \circ {\rm Fr}_2^\vee$, where
\[
{\rm Fr}_1^\vee(x^\gamma y^\lambda \theta^I)=\begin{cases}
x^{\gamma}y^{\lambda/p} \theta^I &\textrm{if } \lambda\in p\mathbb N^n,\\
0 & \textrm{otherwise},
\end{cases}
\]
and
\[
{\rm Fr}^\vee_2(x^\gamma y^\lambda \theta^I)=p^{-q^\vee(\lambda,I)}\sum_k c_k \pi^{|k|}x^\gamma y^{\lambda +kA^T}\theta^I\,.
\]
If $\lambda \notin p \mathbb N^n$, then $\lambda+e_ipA^T\notin p\mathbb N^n$ and thus
\[
d_A^\vee \circ {\rm Fr}_1^\vee (x^\gamma y^\lambda \theta^I)= 0 = {\rm Fr_1}^\vee \circ d_{pA}^\vee(x^\gamma y^\lambda \theta^I)\,.
\]
On the other hand, if $\lambda\in p\mathbb N^n$ then
\begin{align}
d_A^\vee \circ {\rm Fr}_1^\vee (x^\gamma y^\lambda \theta^I)& = (\pi p)^{-1} \sum_i (\lambda A^{-T})_i x^\gamma y^{\lambda/p} \partial_{\theta_i}\theta^I+\sum_i x^\gamma y^{\lambda/p+e_iA^T} \partial_{\theta_i}\theta^I\\
&={\rm Fr}_1^\vee \circ d_{pA}^\vee (x^\gamma y^\lambda \theta^I)\,.
\end{align}
To conclude the proof, we compute
\begin{align}
\pi d_{pA}^\vee \circ {\rm Fr}^\vee_2 (x^\gamma y^\lambda \theta^I)&=p^{-q^\vee(\lambda,I)-1}\sum_k\sum_i c_k \pi^{|k|}(\lambda A^{-T}+k)_i x^\gamma y^{\lambda+kA^T} \partial_{\theta_i}\theta^I\\
&\qquad+p^{-q^\vee(\lambda,I)}\sum_k\sum_i c_k \pi^{|k|}x^\gamma y^{\lambda+kA^T+e_ipA^T}\partial_{\theta_i}\theta^I\\
&= \sum_i\sum_kc_k p^{-q^\vee(\lambda,I-e_i)}\pi^{|k|}(\lambda A^{-T})_i x^\gamma y^{\lambda +kA^T}\partial_{\theta_i}\theta^I\\
&\qquad+\sum_{i}\sum_k p^{-q^\vee(\lambda,I-e_i)}c_k \pi^{|k|}x^\gamma y^{\lambda +kA^T+e_i A^T} \partial_{\theta_i}\theta^I\\
&=\pi{\rm Fr}^\vee_2\circ d_A^\vee (x^\gamma y^\lambda \theta^I)\,,
\end{align}
where we used the definition of the $c_k$ and the fact that $d_{pA}^\vee$ has degree $1$ with respect to the $Q_A^\vee$ grading.
 \end{proof}

Let $\Gamma_p$ denote the $p$-adic gamma function. We refer the reader to \cite{Koblitz84} and \cite{Cohen07} for pedagogical accounts. In what follows, we find it helpful to employ the following notational convention. Given a vector $\alpha=(\alpha_1,\ldots,\alpha_n)\in \mathbb Z_p^n$, we denote by $\Gamma_p(\alpha)$ the product $\Gamma_p(\alpha_1)\cdots \Gamma_p(\alpha_n)$.

\begin{proposition}\label{prop:3.4} Let $A$ be a Berglund-H\"ubsch matrix over $\mathbb F_p$ and let $H(\Fr_A)$, $H(\Fr^\vee_A)$ be operators induced by, respectively, $\Fr_A$ and $\Fr_A^\vee$ on the total cohomology of the bicomplex $(\cB_A,d_A,d_A^\vee)$. Then $H(\Fr_A)$, $H(\Fr^\vee_A)$ act diagonally on the standard monomial basis for the cohomology according to the formulas
\[
H(\Fr_A)[x^\gamma y^\lambda \theta^I]=\pi^{(p-1)|\gamma A^{-1}|}p^{-q(\lambda,I)}\Gamma_p(\gamma A^{-1})[x^\gamma y^\lambda \theta^I]
\]
and
\[
H(\Fr^\vee_A)[x^\gamma y^\lambda \theta^I]=\pi^{(p-1)|\lambda A^{-T}|}p^{-q^\vee(\lambda,I)}\Gamma_p(\lambda A^{-T})[x^\gamma y^\lambda \theta^I]\,.
\]
\end{proposition}

\begin{proof}
Since each term at the RHS of \eqref{eq:3.1} is labeled by  $\gamma + kA \in p \N^n$, we can write
\[
k = [-\gamma A^{-1}] + ps\,,
\]
where $s \in \N^n$ and $[-\gamma A^{-1}] \in \N^n$ is the vector satisfying $0 \leq [-\gamma A^{-1}]_i < p$ and $[-\gamma A^{-1}]A \equiv -\gamma \mod{p}$.
This lets us rewrite the sum in terms of $s$.
In particular, define $\alpha$ so that
\[
\gamma + kA = \gamma + [-\gamma A^{-1}]A + psA = p\alpha + psA\,,
\]
and note that
\[
(\gamma + kA)/p = \alpha + sA\,.
\]
We can then write
\begin{equation}\label{eq:unreduced-frobenius-expansion-with-alpha}
\Fr_A(x^\gamma y^\lambda \theta^I)=p^{-q(\lambda,I)} \sum_{s \in \N^n} c_{[-\gamma A^{-1}] + ps} \pi^{\left|[-\gamma A^{-1}] + ps \right|} x^{\alpha + sA}y^\lambda \theta^I\,.
\end{equation}
Since $\det (A) \mid p - 1$, then $(p-1) \gamma A^{-1} \in \Z^n$, and in fact
\[
[-\gamma A^{-1}] = (p-1) \gamma A^{-1}\,.
\]
Substituting this identity into $p\alpha = [-\gamma A^{-1}] A + \gamma$ produces
\[
(p-1)\gamma A^{-1} A + \gamma = p\alpha\,,
\]
implying that $\alpha = \gamma$. To obtain the result for $H(\Fr)$, we reduce each summand to the standard monomial basis using the relation
\begin{equation}\label{eq:reduction-formula}
x^{\gamma + sA} y^\lambda \theta^I = (-\pi)^{|s|} \left( \gamma A^{-1} \right)_{(k)} x^\gamma y^\lambda \theta^I\,,
\end{equation}
which holds on the level of cohomology, where
\[
(\gamma A^{-1})_{(k)}=\prod_{i=1}^{n} (\gamma A^{-1})_i((\gamma A^{-1})_{i}+1)\ldots((\gamma A^{-1})_{i}+k_{i}-1)\,.
\]
To see this, note that $d_A$ gives us the relation
\[
x^{\gamma+e_{i}A} y^\lambda \theta^I =(-\pi)^{-1}\left(\gamma A^{-1}\right)_{i}x^{\gamma}y^\lambda \theta^I \,.
\]
Iterating, we see that
  \begin{align*}
    x^{\gamma+k_{i}e_{i}A}y^\lambda \theta^I
    = & (-\pi)^{-1}\left((\gamma+(k_{i}-1)e_{i}A)A^{-1}\right)_{i}\, x^{\gamma+(k_{i}-1)e_{i}A}y^\lambda \theta^I \\
    = & (-\pi)^{-2}\left((\gamma A^{-1})_{i}+(k_{i}-1)\right)\left((\gamma A^{-1})_{i}+(k_{i}-2)\right)\, x^{\gamma+(k_{i}-2)e_{i}A}y^\lambda \theta^I \\
    = & (-\pi)^{-k_{i}}\left((\gamma A^{-1})_{i}+(k_{i}-1)\right)\ldots\left(\gamma A^{-1}\right)_{i}\, x^{\gamma}y^\lambda \theta^I ,
  \end{align*}
yielding the reduction formula~(\ref{eq:reduction-formula}). Using this formula and the identity $\pi^{p-1} = -p$ in expansion~(\ref{eq:unreduced-frobenius-expansion-with-alpha}), we obtain
\[
H(\Fr_A)[ x^\gamma y^\lambda \theta^I ] = p^{-q(\lambda,I)} \pi^{|\gamma A^{-1}|} \sum_{s \in \N^n} c_{[-\gamma A^{-1}] + ps} p^{|s|} (\gamma A^{-1})_{(s)} x^\gamma y^\lambda \theta^I\,.
\]
By Proposition 11.6.15 in \cite{Cohen07}, the $p$-adic gamma function satisfies
\[
\Gamma_p(pz - a) = \sum_{s \in \N^n} c_{a+ps} p^{|s|} (z)_{(s)} \,,
\]
which gives the result for $H(\Fr_A)$ if we use $z = \gamma A^{-1}$ and $a = [-(\gamma A^{-1})]$.
The statement for $H(\Fr^\vee_A)$ is proved similarly.
\end{proof}

\begin{remark}
We are interested in using the ``arithmetic'' Frobenius to count number of points of stacks defined over a finite field in the sense of \cite{Behrend93}. To this end, it is natural to shift the Frobenius by an overall power of $p$ equal to the dimension of $Y_{A,G}$ i.e.\ by $p^{n-2}$. Furthermore, as pointed out in Remark \ref{rem:2.20}, the total cohomology of $(\mathcal B_A^G,d_A,d_A^\vee)$ (in the Calabi-Yau setting) computes the Chen-Ruan cohomology of the orbifold $Y_{A,G}(\mathbb C)$. The action of the (geometric) Frobenius on orbifolds is studied in
 \cite{Rose07} where (accounting for the differing definitions of age) the contribution of the twisted sectors is modified by a factor of $p^{-{\rm age}(\lambda)-1}$. In our arithmetic setting, this amounts to twisting the ${\rm Fr}_A$ by the operator $p^{1+\rm age}$, where ${\rm age}$ is the operator diagonalized in the monomial basis with eigenvalues given by age. Taking both contributions into account, we arrive at an overall twist by $p^{{\rm age}+n-1}$ in order to correctly count number of points on our invertible orbifolds.
\end{remark}

\begin{theorem}\label{thm:3.6}
Let $A$ be an $n\times n$ Berglund-H\"ubsch matrix over $\mathbb F_p$ and let $G$ be a subgroup of $G_A$ such that $J\in G\cap G^T$. Let ${\rm ST}_p(A,G)$ be the supertrace of $p^{{\rm age}+n-1}H({\rm Fr}_A)$ acting on $H(\mathcal B_A^G,d_A+d_A^\vee)$. Then
\begin{equation}\label{eq:3.18}
{\rm ST}_p(A,G)=\sum_{\gamma\in G^T,\lambda\in G} (-1)^{\dim(\lambda)+{\rm age}^\vee(\gamma)} p^{{\rm age}(\lambda)+{\rm age}^\vee(\gamma)-1} \Gamma_p(\gamma A^{-1})\delta(\gamma,\lambda)\, ,
\end{equation}
where $\delta(\gamma,\lambda)\in \{0,1\}$ is the dimension of the total cohomology of $\mathcal B_A(\gamma,\lambda)$.
\end{theorem}

\begin{proof}
The total grading obtained from the Hodge bigrading of Remark~\ref{rem:2.18} is $\dim(\lambda)+2{\rm age}(\lambda)-2$ so that 
\begin{equation}\label{eq:3.19}
{\rm ST}_p(A,G)=\sum_{\gamma\in G^T,\lambda\in G}(-1)^{\dim(\lambda)} p^{{\rm age}(\lambda)+n-1}\epsilon(\gamma,\lambda)\delta(\gamma,\lambda)\, ,
\end{equation}
where $\epsilon(\gamma,\lambda)$ is the eigenvalue of $H({\rm Fr}_A)$ on $H(\mathcal B_A(\gamma,\lambda),d_A+d_A^\vee)$ calculated in Proposition \ref{prop:3.4}. Our assumptions imply that $|\gamma A^{-1}|={\rm age}^\vee(\gamma)$ is an integer. Moreover, direct inspection of the cohomology generators for chains and loops shows that $q(\lambda,I)=n$ so that
\begin{equation}\label{eq:3.20}
\epsilon(\gamma,\delta)=(-p)^{{\rm age}^\vee(\gamma)}p^{-n}\Gamma_p(\gamma A^{-1})\,.
\end{equation}
Substituting \eqref{eq:3.20} into \eqref{eq:3.19} easily yields \eqref{eq:3.18}.
\end{proof}

\begin{corollary}\label{cor:3.7}
Let $A$ be an $n\times n$ Berglund-H\"ubsch matrix over $\mathbb F_p$ and let $G$ be a subgroup of $G_A$ such that $J\in G\cap G^T$. Assume further that $JA^{-T}J^T=1$ and $n\ge 3$. Let $X_A(\mathbb F_p)$ be the hypersurface in weighted projective space over $\mathbb F_p$ with weights $w_i=m(JA^{-T})_i$, where $m$ is the smallest integer such that each $w_i$ is an integer, cut out by $W_A(x)=0$. Then ${\rm ST}_p(A,G)$ is congruent modulo $p$ to the number of points of $X_A(\mathbb F_p)$. 
\end{corollary}

\begin{proof}
Modulo $p$, the only non-zero contributions to \eqref{eq:3.18} are labeled by pairs $(\gamma,\lambda)$ that are either of the form $(\gamma,0)$ with ${\rm age}^\vee(\gamma)=1$ or of the form $(0,\lambda)$ with ${\rm age}(\lambda)=1$. Since $n\ge 3$ and $JA^{-T}J^T=1$, then $Z_{A,G}(\mathbb C)$ is a connected Calabi-Yau variety and it follows from Remark \ref{rem:2.20} that there are precisely two such contributions: from $(0,J)$ (corresponding to a generator of $H^{0,0}(Z_{A,G}(\mathbb C))$ and from $(J, 0)$ (corresponding to the class of the Calabi-Yau form of $Z_{A,G}(\mathbb C))$. Hence,
\begin{equation}\label{eq:3.22}
{\rm ST}_p(A,G)\equiv 1 + (-1)^{n-1} \Gamma_p(J A^{-1}) \mod p\,.
\end{equation}
Using the reflection formula for the $p$-adic gamma function in the form 
\begin{equation}\label{eq:3.22bis}
\Gamma_p\left(\frac{a}{d}\right)\Gamma_p\left(1-\frac{a}{d}\right)=(-1)^{p-\frac{a(p-1)}{d}}\,,
\end{equation}
valid whenever $d$ is an integer dividing $p-1$ and $a\in \{1,\ldots,d-1\}$, and the fact that $\Gamma_p(x)$ is congruent mod $p$ to $\Gamma_p(y)$ whenever $x$ and $y$ are (See e.g.\ Sec. 2 Ch. 4 of \cite{Koblitz84}), we obtain
\begin{align*}
1 &=(-1)^{p-(p-1)(J A^{-1})_i}\Gamma_p((J A^{-1})_i)\Gamma_p(1-(J A^{-1})_i)\\
 & \equiv (-1)^{p-(p-1)(J A^{-1})_i}\Gamma_p((J A^{-1})_i)\Gamma_p(p-(p-1)(1-(J A^{-1})_i)) \, \mod p\,.
\end{align*} 
Since $(p-1)J A^{-1}$ is an integer vector, we see from the product representation (ibid.\ ) of $\Gamma_p(p-(p-1)(1-(J A^{-1})_i))$ that
\begin{equation}\label{eq:3.23}
1\equiv \Gamma_p((J A^{-1})_i)((p-1)(J A^{-1})_i)!\, \mod p\, .
\end{equation}
Solving \eqref{eq:3.23} for $\Gamma_p((J A^{-1})_i)$ and substituting into
 into \eqref{eq:3.22}, we arrive at
\begin{equation}
{\rm ST}_p(A,G)\equiv 1 + (-1)^n \binom{p-1}{(p-1)J A^{-1}} \mod p\,,
\end{equation}
where we use the notation $\binom{r}{\beta}=\binom{r}{\beta_1,\ldots,\beta_n}$ for any integer vector $\beta=(\beta_1,\ldots,\beta_n)$ such that $\beta_1+\ldots+\beta_n=r$ to denote multinomial coefficients. On the other hand (a similar calculation can be found e.g.\ in \cite{P}),
\begin{align*}
\#X_A(\mathbb F_p) &\equiv
1+\sum_{x\in \mathbb F_p^n} (W_A(x))^{p-1} \, \mod p\\
&\equiv 1+ \sum_{x\in \mathbb F_p^n} \sum_{|c|=p-1} \binom{p-1}{c} x^{cA}\, \mod p\\
&\equiv 1+ \binom{p-1}{(p-1)JA^{-1}} \sum_{x\in \mathbb F_n} x^{(p-1)J}\, \mod p\\
&\equiv 1+ (-1)^n\binom{p-1}{(p-1)JA^{-1}}\, \mod p\,,
\end{align*}
where for $N$ denoting the number solutions to $W_A(x) = 0$ in $\F_p^n$, we have used the relation $(p-1) \#X_A(\F_p) = N - 1$ along with $N \equiv - \sum_{x \in \F_p^n} W_A(x)^{p-1} \mod p$.
\end{proof}

\section{An Arithmetic Conjecture}

The results in this section apply to Berglund-H\"ubsch matrices in the sense of Definition \ref{def:2.1}. In particular, the restrictive requirement that the determinant divides $p-1$ is assumed.

\begin{conjecture}\label{con:4.1}
Let $A$ be an $n\times n$ Berglund-H\"ubsch matrix over $\mathbb F_p$ let $n\ge 3$, and let $G$ be a subgroup of $G_A$ such that $J\in G\cap G^T$. Assume that $JA^{-T}J^T=1$ and let $X_A$ be the hypersurface in the weighted projective space over $\mathbb F_p$ with weights $w_i=m(JA^{-T})_i$, where $m$ is the smallest integer such that each $w_i$ is an integer, cut out by $W_A(x)=0$. If $ Y_{A,G}$ is the orbifold $[ X_A/(G/\langle J\rangle)]$, then the number of points of any crepant resolution $Z_{A,G}(\mathbb F_p)$ of $ Y_{A,G}(\mathbb F_p)$ is equal to ${\rm ST}_p(A,G)$.
 \end{conjecture}

\begin{remark}
Wan's congruence mirror conjecture \cite{Wan05} implies that if $(X, X')$ is a strong mirror pair of Calabi-Yau manifolds defined over $\mathbb F_p$, then $\# X(\mathbb F_p)$ is congruent to $\# X'(\mathbb F_p)$ modulo $p$. As pointed out in \cites{Wan05,WanFu06}, the notion of a strong mirror pair lacks a rigorous general definition. In \cite{MagyarWhitcher18}, a strong mirror pair is defined as a mirror pair for which the congruence mirror conjecture holds. In any case, the congruence mirror conjecture has been proven \cite{W} for Dwork pencils and their mirrors. It is natural to ask, as done in \cites{MagyarWhitcher18, DKSSVW18, Whitcher21, SalernoWhitcher22}, whether the Berglund-H\"ubsch construction \cite{BH} is compatible with the congruence mirror conjecture -- more precisely, whether $\# Z_{A,G}(\mathbb F_p)$ is congruent to $\# Z_{A^T,G^T}(\mathbb F_p)$ modulo $p$. 

Under the assumptions of Corollary \ref{cor:3.7}, the congruence class modulo $p$ of ${\rm ST}_p(A,G)$ is independent of $G$. This means that Conjecture~\ref{con:4.1} implies the congruence mirror conjecture when $A$ is symmetric. This is consistent with the observation \cite{DKSSVW18}, that Wan's conjecture holds if $A$ is a direct sum of Fermats and loops (no chains). On the other hand, if $A$ is not symmetric , there is no reason to expect ${\rm ST}_p(A,G)$ and ${\rm ST}_p(A^T,G^T)$ to be congruent modulo $p$. Indeed, this is false for $W_A(x)=x_1^2x_2+x_2^3+x_3^3$ and $p=19$: using \eqref{eq:3.22}, ${\rm ST}_{19}(A,\langle J\rangle)\equiv 9 \mod 19$ and ${\rm ST}_{19}(A^T,\langle J\rangle)\equiv 8\mod 19$. In other words, Corollary~\ref{cor:3.7} implies that the congruence mirror conjecture does not apply to the pair of Berglund-H\"ubsch mirror elliptic curves $Z_{A,\langle J\rangle}=X_A$ and $Z_{A^T,\langle J\rangle }=X_{A^T}$.

In conclusion, the congruence mirror symmetry conjecture fails if (as one would expect on physical grounds) non-diagonal Berglund-H\"ubsch mirror pairs are viewed as strong mirror pairs. It is therefore unsurprising that Conjecture \ref{con:4.1} is incompatible with the congruence mirror conjecture in those cases.
\end{remark}

\begin{remark}
Conjecture \ref{con:4.1} can be thought of as an arithmetic analogue of the Crepant Resolution Conjecture. In light of \cite{Yasuda04}, one may wonder whether a more general statement for Gorenstein orbifolds might hold. Since at present the scope of our methods is limited to the Calabi-Yau setting of Borisov's construction, we leave this line of inquiry for future work. 
\end{remark}

 \begin{theorem}\label{thm:4.2}
Let $A$ be a diagonal $n\times n$ Berglund-H\"ubsch matrix over $\mathbb F_p$ such that $JA^{-T}J^T=1$ and $n\ge 3$, and let $G=\langle J \rangle$. Assume that for all $i\le i\le n$ the weights $w_i=m(JA^{-T})_i$, where $m$ is the smallest integer such that each $w_i$ is an integer, are pairwise coprime. Then Conjecture \ref{con:4.1} is true. 
 \end{theorem}

\begin{proof} Since
\[
w_i\le m(JA^{-T})_i\le \det(A)(JA^{-T})_i\le \det(A)\,,
\]
it follows from the Berglund-H\"ubsch condition on $p$ that each $w_i$ is coprime with $p$ and thus we are in the setting of \cite{Goto}. In particular, because all the weights are pairwise coprime, we have that $Y_{A,G}(\mathbb F_p)=Z_{A,G}(\mathbb F_p)$ is a smooth hypersurface in the weighted projective space $\mathbb P^{n-1}(w_1,\ldots,w_n)$. The elements $(\gamma,\lambda)\in G^T\times G$ such that $\delta(\gamma,\lambda)\neq 0$ are of the form: 1) $(0, kJ)$ for $k\in \{1,\ldots,l-1\}$ not divisible by any $A_{ii}$, where $l$ is the least common multiple of the $A_{ii}$, or 2) $(\gamma, 0)$ with $\gamma=(\gamma_1,\ldots,\gamma_n)$ such that each $\gamma_i$ is in $\{1,\ldots,A_{ii}-1\}$ and $|\gamma A^{-1}|$ is an integer. 

The contribution from elements in 1) to \eqref{eq:3.18} is of the form $p^{\text{age}(kJ) - 1}$ for $k$ not divisible by any $A_{ii}$ (since the assumption that the weights $w_i$ are pairwise coprime together with $\gamma A^{-1} \lambda > 0$ implies that $n_\lambda = 0$). Going from $rJ$ to $(r+1)J$ increases the age by one unless $r$ or $r+1$ is divisible by one of the $A_{ii}$, which are excluded from 1), so the total number of elements in 1) is $l - \sum_{i=1}^n (\frac{l}{A_{ii}}-1) - 1 = n - 1$ and the total contribution of these sectors to \eqref{eq:3.18} is $1+p+\ldots+p^{n-2}$. 

On the other hand, each term labeled by $(\gamma, 0)$ is of the form
\begin{equation}
    (-1)^{n+{\rm age}^\vee(\gamma)}p^{{\rm age}^\vee(\gamma)-1}\Gamma_p(\gamma A^{-1})\,.
\end{equation} 
Taking all these contributions into account, \eqref{eq:3.18} specializes to

\begin{equation}\label{eq:4.1}
1+p+\cdots+p^{n-2}+\frac{(-1)^n}{p}\sum_{\gamma} (-p)^{{\rm age}^\vee(\gamma)}\Gamma_p(\gamma A^{-1})\,,
\end{equation}
the sum being extended over all $\gamma$ such that each $\gamma_i\in \{1,\ldots,A_{ii}-1\}$ and ${\rm age}^\vee(\gamma)\in \mathbb Z$. On the other hand, let $\chi \colon \F_p \to \C_p$ be the Teichm\"uller character, and define $\Theta \colon \F_p \to \C_p^\times$ to be the additive Dwork character $x \mapsto \zeta_p^{\chi(x)} = \sum_{s \geq 0} c_s \pi^s \chi^s(x)$, where $\zeta_p$ is a $p$-th root of unity.
For any $t \in \Z$ we can define the Gauss sum
\begin{equation}
G_t = \sum_{\F^\times} \Theta(x) \chi^{t'}(x)\,,
\end{equation}
where $t'$ is chosen to satisfy $t' \equiv t \text{ mod }p - 1$ and $0 < t' \leq p-1$. Adapting the calculations of \cite{Weil49} to the $p$-adic setting by employing the Techmu\"uller and Dwork characters introduced above, we have that the number of points of $Z_{A,G}(\mathbb F_p)$ is equal to
\begin{equation}\label{eq:4.3}
1+p+\cdots+p^{n-2}+\frac{1}{p}\sum_\beta G_{\frac{(p-1)}{A_{11}}\beta_1} \cdots G_{\frac{(p-1)}{A_{nn}}\beta_n}\,,
\end{equation}
where the sum is over all $\beta=(\beta_i,\ldots,\beta_n)$ such that each $\beta_i\in\{1,\ldots,A_{ii}-1\}$ and $\sum_{i=1}^n \frac{\beta_i}{A_{ii}}=0 \mod 1$. Using the special case of the Gross-Koblitz given in \cite{Koblitz84} (taking into account a different sign convention for Gauss sums)
\begin{equation}\label{eq:4.4}
G_{\frac{p-1}{d}\beta} = p \pi^{-\frac{p-1}{d}\beta} \Gamma_p \left( 1- \frac{\beta}{d}\right)\,,
\end{equation}
valid whenever $d$ divides $p-1$,
and identifying indices by setting $\beta_i=A_{ii}-\gamma_i$, we obtain the equality of \eqref{eq:4.1} and \eqref{eq:4.3}.
\end{proof}

\begin{example}
In particular, it follows from Theorem \ref{thm:4.2} that Conjecture \ref{con:4.1} holds for Fermat hypersurfaces i.e.\ the smooth Calabi-Yau hypersurfaces defined by the equation $x_1^n+\cdots+x_n^n=0$ in $\mathbb P^{n-1}(\mathbb F_p)$ for any integer $n\ge 2$. We refer the reader to \cite{GouveaYui95} for an in-depth study of these and related hypersurfaces. 
\end{example}

\begin{theorem}
Conjecture \ref{con:4.1} holds if $n=3$. 
\end{theorem}

\begin{proof}
It is easy to see that $JA^{-T}J^T=1$ implies $A_{ii}\le 6$ for each $i\in \{1,2,3\}$. A direct search of these finitely many cases reveals, up to relabeling, the following 13 matrices which we divide into three types:
\begin{align*}
{\rm type\, I:} &
\left(\begin{smallmatrix} 3 & 0 & 0\\ 0 & 3 & 0\\ 0& 0 & 3\end{smallmatrix}\right)\,, \left(\begin{smallmatrix} 3 & 1 & 0\\ 0 & 2 & 0\\ 0& 0 & 3\end{smallmatrix}\right)\,,\left(\begin{smallmatrix} 2 & 1 & 0\\ 1 & 2 & 0\\ 0& 0 & 3\end{smallmatrix}\right)\,,\left(\begin{smallmatrix} 3 & 1 & 0\\ 0 & 2 & 1\\ 0& 0 & 2\end{smallmatrix}\right)\, , \left(\begin{smallmatrix} 2 & 1 & 0\\ 0 & 2 & 1\\ 1& 0 & 2\end{smallmatrix}\right)\,;\\
{\rm type\, II:} & \left(\begin{smallmatrix} 2 & 0 & 0\\ 0 & 4 & 0\\ 0& 0 & 4\end{smallmatrix}\right)\,, \left(\begin{smallmatrix} 2 & 1 & 0\\ 0 & 2 & 0\\ 0& 0 & 4\end{smallmatrix}\right)\,, \left(\begin{smallmatrix} 4 & 1 & 0\\ 0 & 3 & 0\\ 0& 0 & 2\end{smallmatrix}\right)\,, \left(\begin{smallmatrix} 3 & 1 & 0\\ 1 & 3 & 0\\ 0& 0 & 2\end{smallmatrix}\right)\,, \left(\begin{smallmatrix} 2 & 1 & 0\\ 1 & 2 & 0\\ 0& 0 & 3\end{smallmatrix}\right)\,;\\
{\rm type\, III:} & \left(\begin{smallmatrix} 2 & 0 & 0\\ 0 & 3 & 0\\ 0& 0 & 6\end{smallmatrix}\right)\,, \left(\begin{smallmatrix} 2 & 1 & 0\\ 0 & 3 & 0\\ 0& 0 & 3\end{smallmatrix}\right)\,, \left(\begin{smallmatrix} 3 & 1 & 0\\ 0 & 4 & 0\\ 0& 0 & 2\end{smallmatrix}\right)\,.
\end{align*}
A straightforward application of Theorem \ref{thm:3.6} shows that if $A$ is any matrix of type I, then
\begin{equation}
{\rm ST}_p(A,\langle J\rangle )=1 + p + \left(\Gamma_p\left( \frac{1}{3}\right) \right)^3 - p \left(\Gamma_p\left( \frac{2}{3}\right) \right)^3\,.
\end{equation}
Similarly, if $A$ is any matrix of type II, then 
\begin{equation}
{\rm ST}_p(A,\langle J\rangle )=1 + p + \Gamma_p\left(\frac{1}{2}\right)\left(\Gamma_p\left( \frac{1}{4}\right) \right)^2 - p\Gamma_p\left(\frac{1}{2}\right)\left(\Gamma_p\left( \frac{3}{4}\right) \right)^2\,,
\end{equation}
while 
\begin{equation}
{\rm ST}_p(A,\langle J\rangle )=1 + p + \Gamma_p\left(\frac{1}{2}\right)\Gamma_p\left( \frac{1}{3}\right)\Gamma_p\left(\frac{1}{6}\right) - p\Gamma_p\left(\frac{1}{2}\right)\Gamma_p\left( \frac{2}{3}\right)\Gamma_p\left(\frac{5}{6}\right)
\end{equation}
whenever $A$ is a matrix of type III. Since Theorem \ref{thm:4.2} applies to the diagonal matrices in each type, it remains to show that if $A$ and $A'$ are of the same type, then the elliptic curves $X_A(\mathbb F_p)=Z_{A,\langle J\rangle}(\mathbb F_p)$ and $X_{A}'(\mathbb F_p)=Z_{A',\langle J\rangle}(\mathbb F_p)$ have the same number of points. To see this, consider for instance the last two listed type I elliptic curves with equations $x_2^3x_3+x_3^2x_1+x_1^2 = 0$ and $x_1^2x_2+x_2^2x_3+x_3^2x_1=0$. Since they coincide when restricted to the affine patch $x_2\neq 0$, they are birational, and thus isomorphic. The remaining cases are completely analogous and left to the reader.
\end{proof}

\begin{table}
\caption{Singular weighted diagonal K3 hypersurfaces. Multiplicity refers to the number of singular points of that type on $X_{A}$ . The last column lists the number of $-2$ curves appearing in the minimal resolution of the corresponding singularity\label{table:1}}.

\label{table:weighted_diagonal_k3s}
 \small
 \centering
\begin{tabular}{cccc}

\toprule
$W_A(x)$ & Singularities & Multiplicity & Contribution to $\nu$ \\
\midrule
\addlinespace[0.7em]
 $x_1^2+x_2^3+x_3^7+x_4^{42}$& $A_{7,6}$ & 1 & 6  \\
 &   $A_{3,2}$ & 1 & 2  \\
 &   $A_{2,1}$ & 1 & 1  \\
\midrule
\addlinespace[0.7em]
 
 $x_1^2+x_2^3+x_3^{10}+x_4^{15}$ & $A_{5,4}$ & 1 & 4  \\
 &   $A_{3,2}$ & 2 &  2  \\
 &   $A_{2,1}$ & 3 & 1 \\
 \midrule
\addlinespace[0.7em]
 $x_1^2+x_2^3+x_3^8+x_4^{24} $& $A_{4,3}$ & 1 & 3  \\
 &   $A_{3,2}$ & 2 & 2  \\
 \midrule
\addlinespace[0.7em]
$x_1^2+x_2^4+x_3^5+x_4^{20}$ & $A_{5,4}$ & 2 & 4  \\
 &   $A_{2,1}$ & 1 & 1 \\
 \midrule
\addlinespace[0.7em]
$x_1^2+x_2^3+x_4^{18}+x_5^{18}$ & $A_{3,2}$ & 1 & 2  \\
 &   $A_{2,1}$ & 3 & 1  \\
 \midrule
\addlinespace[0.7em]
$x_1^2+x_2^3+x_3^{12}+x_4^{12}$ & $A_{2,1}$ & 1 & 1  \\
\midrule
\addlinespace[0.7em]
$ x_1^2+x_2^4+x_3^6+x_4^{12}$& $A_{3,2}$ & 2 & 2 \\
 &   $A_{2,1}$ & 2 & 1  \\
 \midrule
\addlinespace[0.7em]
 $x_1^3+x_2^3+x_3^4+x_4^{12} $& $A_{4,3}$ & 3 & 3  \\
 \midrule
\addlinespace[0.7em]
$x_1^3+x_2^4+x_3^4+x_4^{6}$ & $A_{3,2}$ & 4 & 2  \\
 &   $A_{2,1}$ & 3 & 1  \\
 \midrule
\addlinespace[0.7em]
$x_1^2+x_2^5+x_3^5+x_4^{10}$ & $A_{2,1}$ & 5 & 1   \\
\midrule
\addlinespace[0.7em]
 $ x_1^2+x_2^4+x_3^8+x_4^8$ & $A_{2,1}$ & 2 & 1 \\
 \midrule
\addlinespace[0.7em]
 $x_1^3+x_2^3+x_3^6+x_4^6$& $A_{2,1}$ & 3 & 1  \\
\bottomrule
\end{tabular}
\end{table}

\begin{theorem}\label{thm:4.6}
Conjecture \ref{con:4.1} holds for all weighted diagonal K3 surfaces.
\end{theorem}

\begin{proof}
Weighted diagonal K3 surfaces over finite fields are classified in \cite{Goto}. Two of them, the Fermat quartic in $\mathbb P^3$ and the hypersurface $x_1^2+x_2^6+x_3^6+x_4^6=0$ in $\mathbb P(3,1,1,1)$ are smooth and Theorem \ref{thm:4.2} applies to them. The remaining $12$ cases with corresponding lists of singularities and the number of $-2$ curves added in the corresponding minimal resolution are listed in Table \ref{table:1}. In general, as observed in \cite{Goto}, some of the singularities in Table~\ref{table:1} are only defined over suitable field extensions. However, because in our case the coefficients in front of the monomials of $W_A(x)$ are always equal to $1$, these extensions are generated by elements $\gamma$ such that $1+\gamma^m=0$ with $m\in \{1,2\}$ depending on the singularity. On the other hand, inspection of Table \ref{table:1} shows that $\det(A)$, and thus $p-1$, is always divisible by $4$. Hence our assumptions guarantee that $\mathbb F_p$ has primitive fourth roots of unity and thus $\mathbb F_p(\gamma)=\mathbb F_p$. Therefore, all the singularities in Table \ref{table:1} actually occur, just as they do over $\mathbb C$, when $A$ satisfies the assumptions of Conjecture \ref{con:4.1}. Combining the calculations of \cite{Weil49} with \eqref{eq:4.4} as in the proof of Theorem \ref{thm:4.2}, we obtain
\begin{equation}\label{eq:4.8}
\# X_{A}(\mathbb F_p)=1+p+p^2 + \sum_{\gamma} (-p)^{{\rm age}^\vee(\gamma)-1}\Gamma_p(\gamma A^{-1})\,,
\end{equation}
the sum being extended over all $\gamma$ such that each $\gamma_i\in \{1,\ldots,A_{ii}-1\}$ and ${\rm age}^\vee(\gamma)\in \mathbb Z$. We conclude that
\begin{equation}\label{eq:4.9}
\#Z_{A,\langle J\rangle}(\mathbb F_p)=1+(\nu +1)p +p^2+\sum_{\gamma}  (-p)^{{\rm age}^\vee(\gamma)-1}\Gamma_p(\gamma A^{-1})\,,
\end{equation}
 where $\nu$ is the total number of $-2$ curves appearing in $Z_{A,\langle J\rangle}$ viewed as the minimal resolution of $X_{A}$. Since the sum over $\gamma$ in \eqref{eq:4.9} is equal to the contribution from terms labeled by $\lambda=0$ terms to \eqref{eq:3.18}, it remains to check that the contribution from the $\lambda\neq 0$ terms in \eqref{eq:3.18} is exactly $1+(\nu+1)p+p^2$. Since the sum over $\gamma$ in \eqref{eq:4.9} is equal to the contribution from terms labeled by $\lambda=0$ terms to \eqref{eq:3.18}, it suffices to check that the contribution from the $\lambda\neq 0$ terms in \eqref{eq:3.18} is exactly $1+(\nu+1)p+p^2$. It is easy to see that $y_1y_2y_3y_4$ contributes $1$ while $y_1^{A_{11}-1}y_2^{A_{22}-1}y_3^{A_{33}-1}y_4^{A_{44}-1}$ contributes $p^2$. By Remark \ref{rem:2.20}, we know that the remaining $\lambda\neq 0$ terms all correspond to cohomology classes in $H^{1,1}(Z_{A,G}(\mathbb C))$. In particular, there must be precisely $\nu$ of them and they must be of one of two kinds: monomials of the form $y^\lambda$ with ${\rm age}(\lambda)=1$ and monomials of the form $x_i^{\gamma_i}x_j^{\gamma_j}y_r^{\lambda_r}y_s^{\lambda_s}\theta_i\theta_j$ with ${\rm age}(\lambda)=1={\rm age}^\vee(\gamma)$. Each of these $y^\lambda$ monomials contributes $p$ to \eqref{eq:3.18}. The constraint ${\rm age}^\vee(\gamma)=1$ forces each monomial $x_i^{\gamma_i}x_j^{\gamma_j}y_r^{\lambda_r}y_s^{\lambda_s}\theta_i\theta_j$ to be such that $\frac{\gamma_i}{A_{ii}}+\frac{\gamma_j}{A_{jj}}=1$, making the corresponding contribution to \eqref{eq:3.18} equal to
 \begin{equation}\label{eq:4.10}
-p\Gamma_p\left(\frac{\gamma_i}{A_{ii}}\right)\Gamma_p\left(\frac{\gamma_j}{A_{jj}}\right)=-p(-1)^{p-(p-1)\frac{\gamma_i}{A_{ii}}}\,,
 \end{equation}
where the reflection formula \eqref{eq:3.22bis} was used to derive \eqref{eq:4.10}. Inspection of Table \ref{table:1} shows that each $W_A(x)$ has at least two terms with even exponents. Taking again into account that $\det(A)$ divides $p-1$, we conclude that $\frac{p-1}{A_{ii}}$ is necessarily even and thus $p-(p-1)\frac{\gamma_i}{A_{ii}}$ is necessarily odd. Hence, monomials of the form $x_i^{\gamma_i}x_j^{\gamma_j}y_r^{\lambda_r}y_s^{\lambda_s}\theta_i\theta_j$ contribute by $p$ to \eqref{eq:3.18}, concluding the proof.

\end{proof}

\begin{remark}
The results of \cite{Goto} require the assumption that $p$ is coprime to each of the weights of the ambient weighted projective space. In our case, this condition is automatically satisfied by our assumption that $\det(A)$ divides $p-1$. Indeed, the largest prime that appears as a factor of one of the weights in the list of weighted diagonal hypersurfaces in \cite{Goto} is $7$, while the smallest value of $\det(A)$ in the same list is $256$. 
\end{remark}

\begin{example}
We illustrate the proof of Theorem \ref{thm:4.6} by listing all monomials and their contributions to \eqref{eq:3.18} in Table \ref{table:2}. In this case, since $X_{A}$ has three $A_{2,1}$ singularities, two $A_{3,2}$ singularities, and a single $A_{5,4}$ singularity, then $\nu=3+4+4=11$. Correspondingly, there are $(\nu+1)=12$ entries (listed from the third to the fourteenth) in Table \ref{table:2}, each contributing $p$ to \eqref{eq:3.18}.

\begin{table}
\caption{Contributions of each basis element of $H(\mathcal B_A(\gamma,\lambda),d_A+d^\vee_{A^T})$ to \eqref{eq:3.18} when $W_A(x)=x_1^2 + x_2^3 + x_3^{10} + x_4^{15} = 0$. The first column lists the corresponding Hodge numbers of $Z_{A,G}(\mathbb C)$. \label{table:2} 
}
\setlength{\tabcolsep}{0.7em} 
\centering
\small
\begin{tabular}{lll}

\toprule
Hodge Number & Monomial & Contribution to \eqref{eq:3.18}  \\
\midrule
$\left(0, 0\right)$ & $y_{1}^{1} y_{2}^{1} y_{3}^{1} y_{4}^{1}$ & $1$  \\
\addlinespace[0.5em]

$\left(2, 0\right)$ & $x_{1}^{1} x_{2}^{1} x_{3}^{1} x_{4}^{1}  \theta_{1} \theta_{2} \theta_{3} \theta_{4}$ & $- \, \Gamma_p \left( \frac{1}{2} \right) \Gamma_p \left( \frac{1}{3} \right) \Gamma_p \left( \frac{1}{10} \right) \Gamma_p \left( \frac{1}{15} \right)$  \\
\addlinespace[0.5em]

$\left(1, 1\right)$ & $y_{1}^{1} y_{2}^{1} y_{3}^{3} y_{4}^{13}$ & $p$  \\
[0.5em]
 & $y_{1}^{1} y_{2}^{1} y_{3}^{5} y_{4}^{10}$ & $p$  \\
 [0.5em]
 & $y_{1}^{1} y_{2}^{1} y_{3}^{7} y_{4}^{7}$ & $p$  \\
 [0.5em]
 & $y_{1}^{1} y_{2}^{1} y_{3}^{9} y_{4}^{4}$ & $p$  \\
 [0.5em]
 & $y_{1}^{1} y_{2}^{2} y_{3}^{1} y_{4}^{11}$ & $p$  \\
 [0.5em]
 & $y_{1}^{1} y_{2}^{2} y_{3}^{3} y_{4}^{8}$ & $p$   \\
 [0.5em]
 & $y_{1}^{1} y_{2}^{2} y_{3}^{5} y_{4}^{5}$ & $p$   \\[0.5em]
 & $y_{1}^{1} y_{2}^{2} y_{3}^{7} y_{4}^{2}$ & $p$   \\
 [0.5em]
 & $x_{2}^{1} x_{4}^{10} y_{1}^{1} y_{3}^{5} \theta_{2} \theta_{4}$ & $- p \, \Gamma_p \left( \frac{1}{3} \right) \Gamma_p \left( \frac{2}{3} \right)$   \\
 [0.5em]
 & $x_{2}^{2} x_{4}^{5} y_{1}^{1} y_{3}^{5} \theta_{2} \theta_{4}$ & $- p \, \Gamma_p \left( \frac{2}{3} \right) \Gamma_p \left( \frac{1}{3} \right)$  \\
 [0.5em]
 & $x_{1}^{1} x_{3}^{5} y_{2}^{1} y_{4}^{10} \theta_{1} \theta_{3}$ & $- p \, \Gamma_p \left( \frac{1}{2} \right) \Gamma_p \left( \frac{1}{2} \right)$   \\
 [0.5em]
 & $x_{1}^{1} x_{3}^{5} y_{2}^{2} y_{4}^{5} \theta_{1} \theta_{3}$ & $- p \, \Gamma_p \left( \frac{1}{2} \right) \Gamma_p \left( \frac{1}{2} \right)$ \\
 [0.5em]
 & $x_{1}^{1} x_{2}^{1} x_{3}^{3} x_{4}^{13}  \theta_{1} \theta_{2} \theta_{3} \theta_{4}$ & $p \, \Gamma_p \left( \frac{1}{2} \right) \Gamma_p \left( \frac{1}{3} \right) \Gamma_p \left( \frac{3}{10} \right) \Gamma_p \left( \frac{13}{15} \right)$ \\
 [0.5em]
 & $x_{1}^{1} x_{2}^{1} x_{3}^{5} x_{4}^{10}  \theta_{1} \theta_{2} \theta_{3} \theta_{4}$ & $p \, \Gamma_p \left( \frac{1}{2} \right) \Gamma_p \left( \frac{1}{3} \right) \Gamma_p \left( \frac{1}{2} \right) \Gamma_p \left( \frac{2}{3} \right)$ \\
 [0.5em]
 & $x_{1}^{1} x_{2}^{1} x_{3}^{7} x_{4}^{7}  \theta_{1} \theta_{2} \theta_{3} \theta_{4}$ & $p \, \Gamma_p \left( \frac{1}{2} \right) \Gamma_p \left( \frac{1}{3} \right) \Gamma_p \left( \frac{7}{10} \right) \Gamma_p \left( \frac{7}{15} \right)$ \\
 [0.5em]
 & $x_{1}^{1} x_{2}^{1} x_{3}^{9} x_{4}^{4}  \theta_{1} \theta_{2} \theta_{3} \theta_{4}$ & $p \, \Gamma_p \left( \frac{1}{2} \right) \Gamma_p \left( \frac{1}{3} \right) \Gamma_p \left( \frac{9}{10} \right) \Gamma_p \left( \frac{4}{15} \right)$  \\
 [0.5em]
 & $x_{1}^{1} x_{2}^{2} x_{3}^{1} x_{4}^{11}  \theta_{1} \theta_{2} \theta_{3} \theta_{4}$ & $p \, \Gamma_p \left( \frac{1}{2} \right) \Gamma_p \left( \frac{2}{3} \right) \Gamma_p \left( \frac{1}{10} \right) \Gamma_p \left( \frac{11}{15} \right)$ \\
 [0.5em]
 & $x_{1}^{1} x_{2}^{2} x_{3}^{3} x_{4}^{8}  \theta_{1} \theta_{2} \theta_{3} \theta_{4}$ & $p \, \Gamma_p \left( \frac{1}{2} \right) \Gamma_p \left( \frac{2}{3} \right) \Gamma_p \left( \frac{3}{10} \right) \Gamma_p \left( \frac{8}{15} \right)$ \\
 [0.5em]
 & $x_{1}^{1} x_{2}^{2} x_{3}^{5} x_{4}^{5}  \theta_{1} \theta_{2} \theta_{3} \theta_{4}$ & $p \, \Gamma_p \left( \frac{1}{2} \right) \Gamma_p \left( \frac{2}{3} \right) \Gamma_p \left( \frac{1}{2} \right) \Gamma_p \left( \frac{1}{3} \right)$ \\
 [0.5em]
 & $x_{1}^{1} x_{2}^{2} x_{3}^{7} x_{4}^{2}  \theta_{1} \theta_{2} \theta_{3} \theta_{4}$ & $p \, \Gamma_p \left( \frac{1}{2} \right) \Gamma_p \left( \frac{2}{3} \right) \Gamma_p \left( \frac{7}{10} \right) \Gamma_p \left( \frac{2}{15} \right)$ \\
 [0.5em]
$\left(0, 2\right)$ & $x_{1}^{1} x_{2}^{2} x_{3}^{9} x_{4}^{14}  \theta_{1} \theta_{2} \theta_{3} \theta_{4}$ & $- p^{2} \, \Gamma_p \left( \frac{1}{2} \right) \Gamma_p \left( \frac{2}{3} \right) \Gamma_p \left( \frac{9}{10} \right) \Gamma_p \left( \frac{14}{15} \right)$  \\
[0.5em]
$\left(2, 2\right)$ & $y_{1}^{1} y_{2}^{2} y_{3}^{9} y_{4}^{14}$ & $p^{2}$  \\[0.7em]

\bottomrule
\end{tabular}
\end{table}

\end{example}

\section*{Acknowledgements}
It is a pleasure to acknowledge insightful exchanges with Chuck Doran, Doug Haessig, Tristan H\"ubsch, Tyler Kelly, Bong Lian, Daqing Wan, Ursula Whitcher, and Noriko Yui over the long gestation time of this work. We would like to thank the anonymous referee for valuable feedback. M.A.\ was supported in part by VCU Quest Award ``Quantum Fields and Knots: An Integrative Approach.'' Some of the work was carried out while A.P. visited VCU in connection with the Richmond Geometry Meeting funded by NSF grant DMS-2349810, and while A.P. held a Coleman Postdoctoral Fellowship at Queens University, supported by the Natural Sciences and Engineering Research Council (NSERC) of Canada through the Discovery Grant of Noriko Yui.

%\newpage

\end{document}